\documentclass[12pt, a4paper, english]{article}		
\usepackage[T1]{fontenc}
\usepackage{babel}		
\usepackage{fancyhdr}		
\usepackage[utf8]{inputenc}
\usepackage{amsmath}  
\usepackage{amsthm}		
\usepackage{amsfonts}		
\usepackage{amssymb} 
\usepackage{enumerate} 
\usepackage{graphicx} 
\usepackage{indentfirst}
\usepackage{url}
\usepackage[all]{xy}
\newtheorem{teorema}{Theorem}[section]
\newtheorem{lema}[teorema]{Lemma}
\newtheorem{proposicio}[teorema]{Proposition}
\newtheorem{corolari}[teorema]{Corollary}

\theoremstyle{definition}
\newtheorem{definicio}[teorema]{Definition}
\newtheorem{exemple}[teorema]{Example}

\theoremstyle{remark}
\newtheorem*{nota}{Remark}

\DeclareMathOperator{\Ker}{Ker}

\DeclareMathOperator{\Soc}{Soc}

\DeclareMathOperator{\Ret}{Ret}

\newcommand{\bigcdot}{\boldsymbol{\cdot}}

\newcounter{cond}

\title{Addendum/Corrigendum to\\
{\Large On solubility of skew left braces and solutions of the Yang-Baxter equation}}
\author{A. Ballester-Bolinches \thanks{Departament de Matem\`atiques, Universitat de Val\`encia, Av. Vicent Andrés Estellés, 19, 46100 Burjassot, Val\`encia, Spain, \texttt{Adolfo.Ballester@uv.es}, \texttt{Ramon.Esteban@uv.es}, \texttt{Vicent.Perez-Calabuig@uv.es}}\  \and R. Esteban-Romero\addtocounter{footnote}{-1}\footnotemark \and P. Jim\'enez-Seral \thanks{Departamento de Matem\'aticas, Universidad de Zaragoza, Pedro Cerbuna, 12, 50009 Zaragoza, Spain, \texttt{paz@unizar.es}}\and V. P\'erez-Calabuig\addtocounter{footnote}{-2}\footnotemark}

\date{}
\begin{document}
\maketitle

\begin{abstract}
In our previous work: Adv. Math. 455 (2024), no. 109880, solubility of solutions was introduced as an extension of solubility of skew braces in the classification context of non-degenerate solutions of the Yang-Baxter equation. One of our main results (Theorem~C) proved that a skew brace is soluble if, and only if, its associated solution is soluble. A minor step depending on the definition of homomorphism of solutions was overlooked. In this work, proof of Theorem~C is repaired by means of a new class of homomorphisms of solutions: i-homomorphisms of solutions. The importance of this new class is twofold: indecomposable solutions are characterised by means of i-simplicity of solutions, and i-kernels of i-homomorphisms generate ideals in structure skew braces of solutions. Hence, solubility of solutions is redefined as an opposite class of indecomposable solutions. The results obtained with this definition improve our previous outcomes: every soluble solution is proved to have a soluble structure skew brace, and consequently, Theorem~C still holds. Several results stemming from this new analysis are outlined.
\end{abstract}

\emph{Mathematics Subject Classification (2020):  16T25, 20F16, 81R50}  

\emph{Keywords:} Skew left braces, Yang-Baxter equation, solubility, indecomposable solutions, simple solutions.

\section{Introduction}
\emph{In this article, every (set-theoretic) solution $(X,r)$ of the Yang-Baxter equation (YBE) is denoted as $r(x,y) = (\lambda_x(y), \rho_y(x))$ for every $x,y\in X$, and is considered non-degenerate, i.e. $\lambda_x$ and $\rho_x$ are bijective for every $x\in X$.}

There is a wide consensus that the classification problem of solutions goes hand in hand with the development of a structural theory of skew braces. It is well-known that every skew brace provides a solution, and every solution $(X,r)$ has associated two fundamental skew braces that encode its behaviour: the structure skew brace $G(X,r)$ and the permutation skew brace $\mathcal{G}(X,r)$ of~$(X,r)$ ---$(X,r)$ is said to be \emph{injective} if $X$ is embedded in $G(X,r)$ through an injective map $\iota \colon X \rightarrow G(X,r)$. 

One of the main open problems is the classification of \emph{indecomposable solutions}, i.e. solutions which can not be written as a union of two proper subsolutions. This is an outstanding question which is a long way off to be solved (see~\cite{CastelliCatinoPinto19, CedoOkninski23, ColazzoFerraraTrombetti25, JedlickaPilitowska23}). A possible approach consists of analysing solutions that do not admit any non-trivial epimorphic image ---a homomorphism of solutions $(X,r)$ and $(Y,s)$ is a map $f\colon X \rightarrow Y$ which makes compatible the $\lambda$ and $\rho$ components of $r$ and~$s$. These are called \emph{simple solutions}, and structural analysis of associated structure and permutation skew braces have provided significant insight into the classification and construction of such solutions (see \cite{CedoOkninski21,ColazzoJespersKubatVanAntwerpen25}). In particular, the study of simple solutions naturally leads to consider  \emph{simple skew braces}, i.e. those ones that do not admit non-trivial quotients. One can check that injective solutions $(X,r)$ with simple structure skew braces $G(X,r)$ are simple (see Proposition~\ref{prop:G_Xsimple->Xsimple-INJECTIVE}), and simple solutions $(X,r)$ with $|X|\neq 2$ are indecomposable. Nevertheless, the converses of these implications are not true in general: see Example~\ref{ex:contraex-Xsimple-GXnosimple} for the former, and for the latter, every indecomposable solution $(X,r)$ with finite free-square cardinality is not simple  (see~\cite{CedoOkninski23}). Hence, new tools are required to characterise indecomposability.


Following \cite{BallesterEstebanJimenezPerezC24-solubleskewbraces},  in order to classify minimal simple skew braces, solubility of skew braces was defined as an opposite class of simplicity, i.e. skew braces with a rich ideal structure. Similarly, we proposed a definition of solubility of solutions as an antithesis of simple solutions: $(X,r)$ is soluble at $x_t\in X$ if there exist $X_{t} \subseteq \ldots \subseteq X_1 \subseteq X_0 = X$ with $X_{t} = \{x_{t}\}$ such that, for every $1\leq i \leq t$, there exists a homomorphism  $f_i \colon (X,r) \rightarrow (Y_i,s_i)$ satisfying
\begin{itemize}
\item $X_i \in X/\ker f_i$, here $\ker f_i$ is the equivalence relation given by: $x(\ker f_i)y$ if, and only if, $f_i(x) = f_i(y)$;
\item $f_i(X_{i-1}, r|_{X_{i-1}\times X_{i-1}}) = (Z_i,s_i)$ is a twist subsolution of $(Y_i, s_i)$,  i.e. $s_i(u,v) = (v,u)$ for every $u,v\in Z_i$.
\end{itemize}
Our main result in this sense was to prove that solubility of skew braces characterises the soluble character of a solution associated with a skew brace (see \cite[Theorem~C]{BallesterEstebanJimenezPerezC24-solubleskewbraces}).


The motivation of this work comes from~\cite{FerraraTrombettiTsang25}, where the authors ``have spotted an irreparable mistake in the proof of Theorem C'' in~\cite{BallesterEstebanJimenezPerezC24-solubleskewbraces}. We observed that a minor step depending on the definition of homomorphisms of solutions was overlooked. Thus, the main goal of this addendum/corrigendum is to repair the mistake in the proof of Theorem C by revisiting our previous definition of soluble solutions. To this end, we introduce and study what we call i-homomorphisms of solutions. They turn out to be essential to understand indecomposable solutions since i-simplicity is equivalent to indecomposability (Proposition~\ref{prop:indecomp-i-simple}), and i-kernels associated with i-homomorphisms induce ideals in the structure skew brace of a solution (see Theorem~\ref{teo:X0-kerf-ideal}). Then, solubility of solutions naturally emerges as a class of multidecomposable solutions opposed to indecomposable/i-simple ones (see Definition~\ref{def:soluble-solution} in Section~\ref{sec:soluble}). Furthermore, with this revision we improve our previous outcomes: every soluble solution has a soluble structure skew brace (see Theorem~\ref{teo:soluble-solution->GX-soluble}). As a corollary, we extend our previous Theorem~C (see Corollary~\ref{cor:soluble-brace-solution}): for every skew brace $B$ the following statements are equivalent
\begin{itemize}
\item the associated solution $(B,r_B)$ is a soluble solution;
\item $B$ is a soluble skew brace;
\item the associated structure skew brace $G(B,r_B)$  is soluble;
\item the associated permutation skew brace $\mathcal{G}(B,r_B)$ is soluble.
\end{itemize}
We also show that the converse of Theorem~\ref{teo:soluble-solution->GX-soluble} is not true in general (see Example~\ref{ex:multipermut-nosoluble}), but we prove a sufficient condition for the converse: see Theorem~\ref{teo:G_X-soluble+xinI->Xsoluble}. This result yields a new version of Theorem~D in~\cite{BallesterEstebanJimenezPerezC24-solubleskewbraces}. In particular, every non-degenerate solution $(X,r)$ with a soluble skew brace structure, and an element $x\in X$ such that $\lambda_x = \rho_x = \mathrm{id}_X$, is a soluble solution.

\emph{We shall adhere here to our previous notation and preliminaries of~\cite{BallesterEstebanJimenezPerezC24-solubleskewbraces}.}

\section{On homomorphisms of solutions of the YBE}
\label{sec:morphisms}

Let $(X,r)$ be a solution. To avoid notation clash, we may denote $\lambda^r$ and $\rho^r$ to emphasise that they are left and right components of~$r$. We say that $(X,r)$ is a \emph{singleton solution} if $X$ is a singleton set. Recall that $(X,r)$ is a \emph{twist solution} if $r(x,y) = (y,x)$ for every $x,y\in X$, i.e. $\lambda_x = \rho_x = \operatorname{id}_X$ for every $x \in X$. Moreover, $(Z,\left.r\right|_Z)$ is a \emph{subsolution} of $(X,r)$ if $r(Z,Z) = Z\times Z$.

Let $(B,\cdot)$ be a group. Denote by $\operatorname{Sym}_B$ the group of permutations of elements of~$B$. Suppose that there exist a homomomorphism $\lambda$ and an antihomomorphism $\rho$,
\[ \lambda \colon b\in (B,\cdot) \mapsto \lambda_b \in \operatorname{Sym}_B, \quad \rho \colon b \in (B,\cdot) \mapsto \rho_b\in \operatorname{Sym}_B,\]
i.e. a left, respectively right, action of $B$ on $B$, such that  $ab = \lambda_a(b)\rho_b(a)$ for every $a,b\in B$. Theorem~2 of~\cite{LuYanZhu00} ensures that $(B,+,\cdot)$ has a structure of skew brace with a sum defined as $a+b:= a\lambda_{a^{-1}}(b)$, for every $a,b\in B$. Moreover, $\lambda_a$ becomes an automorphism of $(B,+)$ for every $a \in B$, so that $\lambda$ defines an action of $(B,\cdot)$ in $(B,+)$. It turns out that $B$ provides a solution $(B,r_B)$, with $\lambda^{r_B} = \lambda$ and $\rho^{r_B} = \rho$. We call it the \emph{solution associated with~$B$}.

Every solution $(X,r)$ has associated the so-called structure group $G(X,r)$ defined as
\begin{equation}
G(X,r)  = \langle X \mid xy = \lambda^r_x(y)\rho^r_y(x), \,\forall\,x,y\in \! X\rangle. \label{eq:G(X,r)prod} 
\end{equation}
We denote it briefly as $G_X$. According to \cite[Proposition~5]{LuYanZhu00}, $(X,r)$ can be extended to a solution $(\tilde{X}, \tilde{r})$ with $\tilde{X} = X\cup X^{-1}$, and $\left.\lambda^{\tilde{r}}_{x^{\varepsilon}}\right|_X = (\lambda^r)_x^{\varepsilon}$ and $\left.\rho^{\tilde{r}}_{x^{\varepsilon}}\right|_X= (\rho^r)_x^{\varepsilon}$ for every $x\in X$, $\varepsilon = \pm 1$.  Then, proof of Theorem~4 in~\cite{LuYanZhu00} shows that $(\tilde{X}, \tilde{r})$ can be extended to a solution $(G_X, r_{G_X})$ 
so that $w w' = \lambda^{G_X}_w(w')\rho^{G_X}_{w'}(w)$ for every $w, w'\in G_X$. Applying Theorem~2 of~\cite{LuYanZhu00}, this defines a sum in $G_X$ providing a skew brace structure $G_X = (G_X,+,\cdot)$, and the defining relations~\eqref{eq:G(X,r)prod} are translated to the additive group as
\begin{equation}
\label{eq:G(X,r)sum}
(G_X,+) = \langle X \mid x + \lambda^r_x(y) = \lambda_x^r(y) + \lambda_{\lambda_x^r(y)}^r(\rho^r_y(x))\rangle.
\end{equation}
We call $(G_X,+,\cdot)$ the \emph{structure skew brace} of solution~$(X,r)$. Theorem~4 in~\cite{LuYanZhu00} states that $r_{G_X}(\iota \times \iota) = (\iota \times \iota)r$,  where $\iota\colon x\in X \mapsto \iota(x) = x\in G(X)$ is a canonical map (not necessarily injective) such that the following diagram commutes
\[ 
\xymatrix{X \times X \ar[d]_{\iota\times \iota} \ar[r]^{r}  & X \times X  \ar[d]^{\iota \times \iota} \\ G_X \times G_X  \ar[r]_{r_{G_X}} & G_X \times G_X}
\]
Indeed, $\iota$ can be extended to $\iota\colon \tilde{X}\rightarrow G_X$, with $\iota(x^{-1}) = \iota(x)^{-1}$ for every $x\in X$ (we also denote it by~$\iota$ for the sake of simplicity), so that $r_{G_X}(\iota \times \iota) = (\iota \times \iota)\tilde{r}$ also commutes. Therefore, for every $x,y \in X$  and $\varepsilon, \varepsilon' = \pm 1$, it follows that
\begin{align}
\label{eq:lambda-en-GX}
& \lambda^{G_X}_{x^{\varepsilon}}(y) = \lambda^{G_X}_{\iota(x)^{\varepsilon}}(\iota(y)) = \iota\left(\lambda_{x^{\varepsilon}}^{\tilde{r}}(y)\right) = \iota\left((\lambda^r)_x^{\varepsilon}(y)\right) = (\lambda^r)_x^{\varepsilon}(y) \in G_X; \\
\label{eq:rho-en-GX}
& \rho^{G_X}_{x^{\varepsilon}}(y) = \rho^{G_X}_{\iota(x)^{\varepsilon}}(\iota(y)) = \iota\left(\rho_{x^{\varepsilon}}^{\tilde{r}}(y)\right) = \iota\left((\rho^r)_x^{\varepsilon}(y)\right) = (\rho^r)_x^{\varepsilon}(y) \in G_X.
\end{align}

Moreover, $\lambda$ and $\rho$ maps of $G_X$ satisfy that $ww' = \lambda_w(w')\rho_{w'}(w)$ for every $w, w'\in G_X$. Thus, it holds that
\begin{align}
ww_1\cdots w_n & = \lambda_{w}(w_1)\rho_{w_1}(w)w_2\cdots w_n \nonumber \\
& = \lambda_{w}(w_1)\lambda_{\rho_{w_1}(w)}(w_2\cdots w_n) \rho_{w_2\cdots w_n}(\rho_{w_1}(w))  \nonumber \\
& = \lambda_{w}(w_1)\lambda_{\rho_{w_1}(w)}(w_2\cdots w_n) \rho_{w_1w_2\cdots w_n}(w); \nonumber
\end{align}
and it follows that
\begin{equation}
\label{eq:lambda-prod-GX}
\lambda_{w}(w_1\cdots w_n) = \lambda_{w}(w_1)\lambda_{\rho_{w_1}(w)}(w_2\cdots w_n)
\end{equation}
for every $w,w_1, \ldots, w_n \in G_X$. Similarly, since $\rho$ is an antihomomorphism:
\begin{equation}
\label{eq:rho-prod-GX}
\rho_{w}(w_1\cdots w_n) = \rho_{\lambda_{w_n}(w)}(w_1\cdots w_{n-1})\rho_{w}(w_n)
\end{equation}
for every $w,w_1, \ldots, w_n \in G_X$.
%

\bigskip

A \emph{homomorphism} of skew braces $B_1$ and $B_2$ is a map $f\colon B_1\rightarrow B_2$ such that $f(a+b) = f(a) + f(b)$ and $f(ab) = f(a)f(b)$ for every $a,b\in B_1$. As usual, we say that $f$ is \emph{epimorphism} if it is surjective. It turns out that the kernel of~$f$, $\Ker f = \{a \in B_1\mid f(a) = 0\in B_2\}$, is an ideal of~$B_1$. A skew brace $B$ is said to be \emph{simple} if it has exactly two ideals: $\{0\}$ and~$B$.

Following Theorem~4 of~\cite{LuYanZhu00}, the pair $(G_X,\iota)$ of a solution $(X,r)$ satisfies the following universal property: if $B$ is a skew brace and $f\colon X \rightarrow B$ is a map satisfying
\[ \xymatrix{X \times X \ar[d]_{f\times f} \ar[r]^{r}  & X \times X  \ar[d]^{f\times f} \\ B \times B  \ar[r]_{r_B} & B \times B}\]
then there exists a unique homomorphism $g\colon G_X \rightarrow B$ such that the following diagrams commute
\[ \xymatrix{X \ar[r]^{\iota} \ar[d]_{f} & G_X \ar[dl]^{g} \\B & } \qquad \xymatrix{G_X \times G_X \ar[d]_{g\times g} \ar[r]^{r_{G_X}}  & G_X \times G_X  \ar[d]^{g\times g} \\ B \times B  \ar[r]_{r_B} & B \times B}\]
In particular, given a skew brace $B$, if we take $\operatorname{id}\colon B \rightarrow B$ the identity map, then there exists a unique homomorphism $\pi_B\colon G_B \rightarrow B$ such that the following diagrams commute in both directions:
\[ \xymatrix{G_B \times G_B \ar@/^/[d]^{\pi_B\times \pi_B} \ar[r]^{r_{G_B}}  & G_B \times G_B  \ar@/^/[d]^{\pi_B\times \pi_B} \\ B \times B \ar@/^/[u]^{\iota \times \iota}  \ar[r]_{r_B} & B \times B \ar@/^/[u]^{\iota\times \iota}}\]

\begin{definicio}
Let $B$ be a skew brace. We call $\pi_B\colon G_B \rightarrow B$ the \emph{projection epimorphism of $G_B$ onto $B$}. It is the unique skew brace homomorphism satisfying that $\pi_B\iota = \mathrm{id}_B$, where $\iota\colon B \rightarrow G_B$ is the canonical map.
\end{definicio}

%

\begin{nota}
\label{nota:ker-iota-brace}
Let $B$ be a skew brace. Elements in $B$ can be seen in $G_B$ by means of $b\mapsto \iota(b) \in G_B$. In order to avoid notation clash between operations in $B$ and $G_B$, we denote $\bar{b}:= \iota(b) \in G_B$ for every element $b\in B$. Moreover, if $\bar{b}_1 = \iota(b_1) = \iota(b_2) = \bar{b}_2\in G_B$, for some $b_1,b_2\in B$, then $\bar b_1^{-1}\bar b_2 = 0 \in \Ker \pi_B$. Therefore, $\iota$ is injective as
\[ 0 = \pi_B(\bar b_1^{-1}\bar b_2) = \pi_B(\bar b_1)^{-1}\pi_B(b_2) = b_1^{-1}b_2.\]
\end{nota}

\begin{lema}
\label{lema:kerg-socle}
Let $B$ be a skew brace and let $\pi_B\colon G_B \rightarrow B$ be its associated projection epimorphism. Then, $\Ker \pi_B$ is an ideal of $G_B$ contained in $\Soc(G_B)$. In particular, $\langle \bar{0}\rangle_{\bigcdot}$ is a central ideal of $G_B$ which is isomorphic to an infinite cyclic group as an abelian skew brace.
\end{lema}

\begin{proof}
We write just $\lambda$ and $\rho$ (resp. $\lambda^G$ and $\rho^G$) for the left and right actions in~$B$ (resp.~$G_B$).

It suffices to check that $\lambda^{G}_w = \rho^{G}_w = \mathrm{id}_{G_B}$ for every $w\in \Ker \pi_B$ (see \cite[Proposition~5.3]{BallesterEstebanPerezC24-JH} for instance). Let $w\in \Ker \pi_B$, and write $w = \bar b_1^{\varepsilon_1}\cdots \bar b_r^{\varepsilon_r}$, for some $b_1, \ldots, b_n\in B$ and $\varepsilon_i = \pm 1$, $1\leq i \leq r$. Thus, by definition of $\pi_B$, it turns out that $\pi_B(w) = b_1^{\varepsilon_1}\cdots b_n^{\varepsilon_n} = 0\in B$. Then, it follows that
\begin{align*}
 \lambda^{G}_w(\bar b) & = \lambda^{G}_{\bar b_1^{\varepsilon_1}}\cdots \lambda_{\bar b_n ^{\varepsilon_n}}^{G}(\bar b) =  \lambda_{\bar b_1^{\varepsilon_1}}^{G}\cdots \lambda_{\bar b_{n-1}^{\varepsilon_{n-1}}}^{G}(\lambda_{\bar b_n^{\varepsilon_n}}^{G}(\bar b)) =  \lambda_{\bar b_1^{\varepsilon_1}}^{G}\cdots \lambda_{\bar b_{n-1}^{\varepsilon_{n-1}}}^{G}(\overline{\lambda_{b_n}^{\varepsilon_n}(b)})\\
& = \cdots \text{ (applying~\eqref{eq:lambda-en-GX} iteratively)} = \overline{\lambda_{b_1}^{\varepsilon_1}\cdots \lambda_{b_n}^{\varepsilon_n}(b)} =  \overline{\lambda_{b_1^{\varepsilon_1}\cdots b_n^{\varepsilon_n}}(b)} \\
& = \overline{\lambda_{\pi_B(w)}(b)} = \overline{\lambda_0(b)} = \bar b
\end{align*}
for every $b\in B$. Hence, for every $w\in \Ker \pi_B$, $\left.\lambda_w^G\right|_{\iota(B)} = \mathrm{id}_{\iota(B)}$.

Since $\Ker \pi_B$ is an ideal of $G_B$, it is $\lambda^G$ and $\rho^G$-invariant. Therefore, for every $w\in \Ker \pi_B$ and $b\in B$,  by~\eqref{eq:lambda-prod-GX}, it follows that
\[ 0 = \lambda^G_w(\bar b^{-1}\bar b) = \lambda_w^G(\bar b^{-1})\lambda_{\rho^G_{\bar b^{-1}}(w)}^G(\bar b) =  \lambda_w^G(\bar b^{-1}) \bar b\]
as $\rho^G_{\bar b^{-1}}(w)\in \Ker \pi_B$. Thus, $\lambda_w^G(\bar b^{-1}) = \bar b^{-1}$. Hence, using~\eqref{eq:lambda-prod-GX} iteratively, it follows that $\lambda_w^G(w') = w'$ for every $w \in \Ker\pi_B$ and $w'\in G_B$, i.e. $\lambda_w^{G} = \mathrm{id}_{G_B}$ for every $w\in \Ker \pi_B$.

Analogously, using~\eqref{eq:rho-en-GX} and~\eqref{eq:rho-prod-GX}, it also holds that $\rho_w^G = \mathrm{id}_{G_B}$ for every $w\in \Ker \pi_B$.

Finally, take $0\in B$. It holds that $\lambda_0  = \rho_0 = \mathrm{id}_B$, and $\lambda_b(0) = \rho_b(0) = 0$ for every $b\in B$. As a consequence, according to~\eqref{eq:G(X,r)prod} and~\eqref{eq:G(X,r)sum}, it turns out that $\bar{b}\bar{0} = \bar{0}\bar{b}$, and $\bar{b} + \bar{0} = \bar 0 + \bar b$, for every $b\in B$. Hence, $\langle \bar 0 \rangle_{\bigcdot}$ is an ideal which is an abelian skew brace isomorphic to an infinite cyclic group, and $\langle \bar  0 \rangle_{\bigcdot}$ is central in~$B$.
\end{proof}

Observe that if $B$ is a simple skew brace, then every ideal $I$ in $G_B$ satisfies either $I \subseteq \Ker \pi_B$ or $\pi_B(I) = B$. On the other hand, we have the following result as a corollary.

\begin{corolari}
\label{cor:G_B-maiSimple}
For every skew brace $B$, $G_B$ is not simple.
\end{corolari}

\begin{proof}
Assume that $G_B$ is simple for some skew brace~$B$. By Lemma~\ref{lema:kerg-socle}, it follows that $G_B = \langle \bar 0 \rangle_{\bigcdot} = \langle 0 \rangle_+$. Therefore, $G_B$ is an abelian skew brace isomorphic to an infinite cyclic group which is not simple, a contradiction.
\end{proof}

%
%


\bigskip

The notion of a homomorphism of solutions was firstly introduced in~\cite{Gateva-IvanovaMajid07}. Let $(X,r)$ and $(Y,s)$ be solutions. A map $f\colon X \rightarrow Y$ is said to be a \emph{homomorphism} of solutions if it satisfies that $(f\times f) r = s (f \times f)$; or equivalently, if $f(\lambda^r_x(y)) = \lambda^s_{f(x)}(f(y))$ and $f(\rho^r_y(x)) = \rho^s_{f(y)}(f(x))$ for every $x,y\in X$. We call $f$ an \emph{epimorphism} if it is surjective.

\begin{nota}
\label{nota:tots_hom_epi}
Let $f\colon X \rightarrow Y$ be a homomorphisms of solutions $(X,r)$ and~$(Y,s)$. Since $f$ makes compatible $\lambda$ and $\rho$ maps of $r$ and~$s$, then $(f(X), \left.s\right|_{f(X)})$ is also a solution. 
\end{nota}

If $f\colon X \rightarrow Y$ is a homomorphism of solutions $(X,r)$ and $(Y,s)$, then we can consider an equivalence relation in $X$ given by the \emph{kernel relation of~$f$}: $x (\ker f) y$ if, and only if, $f(x) = f(y)$. Observe that $f$ is compatible with the relations defined in~\eqref{eq:G(X,r)prod} and~\eqref{eq:G(X,r)sum}. Thus, 
$f$ can be extended into an homomorphism of structure skew braces $\bar{f} \colon G_X \rightarrow G_Y$ so that $\bar{\iota}f = \bar{f}\iota$ commutes:
\[
\xymatrix{X \ar[d]^{\iota} \ar[r]^{f}  & Y \ar[d]^{\bar{\iota}} \\ G_X \ar[r]_{\bar{f}}  & G_Y }
\]
i.e. $\bar{f}(x) = \bar f(\iota(x)) = \bar{\iota}(f(x)) = f(x) \in G_Y$ for every $x\in X$. It follows that if $x (\ker f) y$, then $x^{-1}y\in \Ker \bar f$ as
\begin{equation}
\label{eq:x(ker f)y->Kerf}
f(x) = f(y) \Rightarrow \bar{f}(x^{-1}y) = \bar{f}(x)^{-1}\bar f(y) = f(x)^{-1}f(y) = 0 \in G_Y
\end{equation}
Recall that in a skew brace $B$, $\lambda_{b^{-1}}(-b)^{-1} = b$ for every $b\in B$. Thus, with the previous conditions, if $f(x) = f(y)$, then 
\[ f(x) = \lambda^{G_Y}_{f(x)^{-1}}(-f(y))^{-1} = f(\lambda^{G_X}_{x^{-1}}(-y)^{-1}),\]
and therefore, by~\eqref{eq:x(ker f)y->Kerf}, $x-y = x\lambda^{G_X}_{x^{-1}}(-y)$ with $x(\ker f) \lambda^{G_X}_{x^{-1}}(-y)^{-1}$.

\begin{lema}
\label{lema:kerbarf}
Let $f \colon X \rightarrow Y$ be an epimorphism of solutions $(X,r)$ and $(Y,s)$, with induced homomorphism $\bar f\colon G_X \rightarrow G_Y$. Then, $\Ker \bar f$ is the ideal generated 
\[ K:= \{x^{-1}y \mid \text{$x,y \in X$ with $x (\ker f) y$}\}.\]
\end{lema}

\begin{proof}
Write $I_K$ the ideal in $G_X$ generated by~$K$. Clearly, \eqref{eq:x(ker f)y->Kerf} yields one inclusion. For the other inclusion, take 
\[ R_{\mathsf{X}} := \{ \mathsf{u}\mathsf{v}\rho_{\mathsf{v}}^{\mathsf{r}}(\mathsf{u})^{-1}\lambda_{\mathsf{u}}^{\mathsf{r}}(\mathsf v)^{-1} \mid \mathsf{u,v} \in \mathsf{X}\},\] for each $\mathsf{X}\in \{X,Y\}$ with respectively $\mathsf{r}\in \{r,s\}$. Since $f$ is surjective and it makes $\lambda$ and $\rho$ maps of $r$ and $s$ compatible, it follows that for every $w\in R_Y$, there exists $u \in R_X$ such that $\bar f(u) = w$.


Let $w_0 \in \Ker \bar f$ and write $w_0 = x_1^{\varepsilon_1}\cdots x_n^{\varepsilon_n}$ for some $x_i \in X$, $\varepsilon_i =\pm 1$, $1\leq i \leq n$. Then,
\[ 0 = \bar f(w_0) = f(x_1)^{\varepsilon_1}\cdots f(x_n)^{\varepsilon_n} \in G_Y.\]
Thus, there exist $v_1, \ldots, v_m \in G_Y$ and $w_1, \ldots, w_m \in R_Y$ such that the word
\[
w':= f(x_1)^{\varepsilon_1} \cdots f(x_n)^{\varepsilon_n} (v_1 w_1 v_1^{-1}) \cdots (v_m w_m v_m^{-1})
\]
reduces recursively to $0$
by successive cancellations of subwords of the form either $yy^{-1}$ or $y^{-1}y$, with $y \in Y$.

Since $f$ is surjective, for each $1 \leq i \leq m$, there exists a word $z_i \in G_X$ such that
\[
\mathrm{length}(z_i) = \mathrm{length}(v_i)\quad\text{and} \quad \bar f(z_i) = v_i,
\]
where each letter of $z_i$ maps under $\bar f$ to the corresponding letter of $v_i$. Moreover, for every $1\leq i \leq m$, there exists $u_i \in R_X$ such that $f(u_i)=w_i$.

Thus, $w_0':= w_0(z_1u_1z_1^{-1})\cdots (z_mu_mz_m^{-1}) = w_0 \in G_X$ and $\bar f(w_0') = w' = 0 \in G_Y$, where $\mathrm{length}(w_0') = \mathrm{length}(w')$, and each letter of $w_0'$ maps under $\bar f$ to the corresponding letter of $w'$. Therefore, in $G_X/I_K$, the word $w_0'I_K$ recursively reduces to $0I_K$ following the same reduction procedure than $w'$: there is a one-to-one correspondence between each cancellation of the form $y^{-1}y \in G_Y$ with $x^{-1}z \in X$, where $f(x) = f(z) = y$, and similarly for $yy^{-1}$.

Hence, $w_0 = w_0'\in G_X$ and $w_0' \in I_K$, so that the other inclusion also holds.
\end{proof}

A non-singleton solution $(X,r)$ is said to be \emph{simple} if for every epimorphism $f\colon (X,r) \rightarrow (Y,s)$, either $f$ is bijective or $f$ is a constant map. In this context, it is natural to analyse the role of simplicity of structure skew braces.

\begin{proposicio}
\label{prop:G_Xsimple->Xsimple-INJECTIVE}
Let $(X,r)$ be a non-singleton injective solution with a simple structure skew brace~$G_X$. Then, $(X,r)$ is simple.
\end{proposicio}

\begin{proof}
Assume that $\iota \colon X \rightarrow G_X$ is injective and $G_X$ is a simple skew brace. Let $f\colon (X,r) \rightarrow (Y,s)$ be an epimorphism of solutions with induced epimorphism $\bar f\colon G_X \rightarrow G_Y$. Since $G_X$ is simple, either $\Ker \bar f = G_X$ or $\Ker \bar f = 0$. The former is not possible, as the structure skew brace of a solution can not be a trivial zero skew brace. Therefore, $\Ker \bar f = 0$. 

If $f$ is not a bijection, then there exist $x,y\in X$ such that $f(x) = f(y)$. Thus, $x^{-1}y\in \Ker \bar f = 0$. Hence, $\iota(x) = \iota(y) \in G_X$, and therefore, $x = y \in X$, as $\iota$ is injective.
\end{proof}

The following example shows that the converse is false in general.

\begin{exemple}
\label{ex:contraex-Xsimple-GXnosimple}
Consider the set $X = \{1,2,3\}$. We define a Lyubashenko solution, i.e a solution given by two commutative permutations in $\Sigma_X$: $r(x,y) = (\sigma(y), \tau(x))$, where $\sigma = (1,2,3), \tau = (1,3,2) \in \Sigma_3$. Since $\tau = \sigma^{-1}$, one can check that $(X,r)$ is involutive, i.e. $r^{-1} = r$, and therefore,
\[ (G_X,+) = \langle X \mid x+y = y+x, \ x,y\in X\rangle\]
is a free abelian group over~$X$. On the other hand, $(G_X,\cdot)$ is given by 
\[ (G_X,\cdot) = \langle X \mid 1\cdot 1 = 2 \cdot 3, \  2\cdot 2 = 3\cdot 1, \  3\cdot 3 = 1\cdot 2 \rangle.\]
As a consequence, it follows that $(X,r)$ is injective as $\iota\colon X \rightarrow G_X$ is injective. It also holds that $1^3 = 2^3 = 3^3 \in G_X$ (e.g. from $2 = 1^{-1}\cdot 3 \cdot 3$, we can see that $1^3 = 3^3$). 

Take $u:= 1^3$. It turns out that $\langle u \rangle_{\bigcdot}$ is a normal subgroup as $xux^{-1} = x^{-1}ux = u$ for every $x\in X$. Moreover, applying~\eqref{eq:lambda-en-GX},
\begin{align*}
 1^3 & = 1\cdot (1 + \lambda_1^{G_X}(1)) = 1\cdot (1 + \lambda_1^r(1)) = 1\cdot (1 + \sigma(1)) = 1 \cdot (1+2) \\
 & = 1 \cdot 1 - 1 + 1 \cdot 2 = 1 + 2 -1 + 1 + \lambda_1^r(2) = 1 + 2 + \sigma(2) = \\
 & = 1 + 2 + 3,
\end{align*}
and $\lambda_u^G(x) = (\lambda_1^r)^3(x) = \sigma^3(x) = x$ for every $x\in X$. Thus, $\lambda_u^G = \mathrm{id}_G$, and therefore, $I:= \langle u\rangle_+ = \langle u \rangle_{\bigcdot}$ is a trivial subbrace of $G_X$, which is a normal subgroup both in $(G_X,\cdot)$ and $(G_X,+)$. Observe that, $I$ is $\lambda^G$-invariant as $\lambda_w(1+2+3) = 1+2+3$ for every $w\in G_X$. Hence, $0\neq I \neq G_X$ is an ideal of $G_X$, i.e. $G_X$ is not a simple skew brace.

On the other hand, assume that $f \colon (X,r) \rightarrow (Y,s)$ is an epimorphism with $Y = \{y_1,y_2\}$. Without loss of generality, suppose that $f(1) = f(2) = y_1$ and $f(3) = y_2$. Then, 
\[ \begin{array}{l}
 y_1 = f(2) = f(\lambda^r_1(1)) = \lambda^s_{y_1}(y_1) \\
 y_1 = f(1) = f(\lambda^r_1(3)) = \lambda^s_{y_1}(y_2)
 \end{array}\]
Therefore, $\lambda^s_{y_1}$ is not bijective and so, $(Y,s)$ is not a non-degenerate solution. Hence, for every epimorphism $f\colon (X,r) \rightarrow (Y,s)$, either $f$ is an isomorphism or $|Y| = 1$; that is, $(X,r)$ is a simple solution.
\end{exemple}

\section{On i-homomorphisms of solutions}
\label{sec:imorphisms}

Previous results show that simplicity of solutions exhibits a rigid behaviour that is not reflected in the ideal structure of the structure skew brace associated with a solution. Therefore, it is natural to consider a more comprehensive notion of homomorphism that enables us to get a tighter grip over 
the associated structure skew brace of a solution.


\begin{definicio}
\label{def:i-homomorf}
Let $(X,r)$ and $(Y,s)$ be solutions. We call a homomorphism of solutions $f\colon X \rightarrow Y$  an \emph{i-homomorphism}, if there exists an equivalence class $X_0\subseteq X$ of $\ker f$ such that for every equivalence class $Z$ of $\ker f$, it holds $r(X_0,Z) = Z \times X_0$ and $r(Z,X_0) = X_0\times Z$. We say $X_0$ is an \emph{i-kernel of~$f$}, and $f$ is said to be an \emph{i-epimorphism} if $f$ is also surjective.
\end{definicio}

The importance of i-homomorphisms of solutions is twofold: on the one hand, we will see that they turn out to be the right tool to characterise indecomposability, and on the other hand, i-kernels of i-homomorphisms will be associated with ideals of structure skew braces of solutions.


Recall that a solution $(X,r)$ is said to be \emph{decomposable} if there exist proper subsets $Y, Z \subseteq X$ such that $X = Y \cup Z$, and $r(Y,Y) = Y \times Y$ and $r(Z,Z) = Z \times Z$. Otherwise, if no such decomposition can be considered, $(X,r)$ is said to be \emph{indecomposable}. It is well-known that the only decomposable simple solution is the twist solution on two elements, and out of it, simplicity of solutions is a stronger property that does not characterise indecomposability.

The next proposition shows that i-homomorphisms characterise indecomposable solutions in terms of an associated \emph{i-simplicity}.

\begin{definicio}
Let $(X,r)$ be a non-singleton solution. We define $(X,r)$ to be \emph{i-simple} if $X$ is the unique i-kernel of every i-homomorphism of solutions $f\colon X \rightarrow Y$.
\end{definicio}


\begin{proposicio}
\label{prop:indecomp-i-simple}
Let $(X,r)$ be a solution such that $X$ is not a singleton set. Then, $(X,r)$ is indecomposable if, and only if, it is i-simple.
\end{proposicio}

\begin{proof}
Assume that $(X,r)$ is indecomposable, and let $f\colon X \rightarrow Y$ be a i-homomorphism between $(X,r)$ and a solution $(Y,s)$. Take $X_0\subseteq X$ an i-kernel of~$f$, and call $Z = X \setminus X_0$. Thus, $r(X_0,Z) = Z \times X_0$ and $r(Z,X_0) = X_0\times Z$. By the non-degenerate property of~$r$, it follows that $r(X_0, X_0) = X_0 \times X_0$ and $r(Z,Z) = Z \times Z$. Therefore, $X_0 = X$ as $(X,r)$ is indecomposable.

Suppose that $(X,r)$ is i-simple. Take $X = Y \cup Z$ for some $Y,Z \subseteq X$, such that $r(Y,Y) = Y \times Y$ and $r(Z,Z) = Z\times Z$. By the non-degenerate property of $(X,r)$, $r(Y,Z) = Z\times Y$ and $r(Z,Y) = Y\times Z$. Take the twist solution $(\{a,b\}, \tau)$ over $\{a,b\}$, and define $f\colon X \rightarrow \{a,b\}$ as $f(Y) = \{a\}$ and $f(Z) = \{b\}$. It follows that $f$ is an i-homomorphism of solutions with i-kernels $Y$ and~$Z$. By the i-simplicity of $(X,r)$, either $Y = X$ or $Z = X$. Hence, $(X,r)$ is indecomposable.
\end{proof}

Observe that if $(X,r)$ is an indecomposable solution, then the identity $\operatorname{id}_X\colon (X,r) \rightarrow (X,r)$ is an example of a homomorphism of solutions which is not an i-homomorphism. 

Now we turn our attention to i-kernels. Firstly, the following example shows that i-homomorphisms do not generate a unique i-kernel.

\begin{exemple}
\label{ex:multiples-kernels}
Let $(X,r)$ and $(Y,s)$ be twist solutions of the same cardinality. A permutation $\sigma \colon X \rightarrow Y$ is an i-homomorphism of solutions such that for every $x\in X$, $\{x\}$ is an i-kernel of~$\operatorname{id}_X$. And vice versa, assume that $f\colon X \rightarrow Y$ is an i-epimorphism of solutions $(X,r)$ and $(Y,s)$ such that $\{x\}$ is a kernel of~$f$, for every $x\in X$. Then, $f$ is bijective, and both $(X,r)$ and $(Y,s)$ are twist solutions.
\end{exemple}

We can see homomorphisms of skew braces as a particular case of i-homomorphisms between associated solutions.

\begin{exemple}
\label{ex:homomorfisme-brides}
Let $f \colon B_1 \rightarrow B_2$ be a homomorphism of skew braces $B_1$ and~$B_2$, and call $J:= \Ker f$, which is an ideal of~$B_1$. Then, $f$ can be seen also as an i-homomorphism of solutions $(B_1,r_{B_1})$ and $(B_2, r_{B_2})$ with i-kernel $J\subseteq B_1$. Indeed, $J$ induces a partition of $B_1$ such that each subset of the partition is of the form $bJ$ for some $b\in B_1$, and it is an equivalence class of elements with the same image under~$f$. By definition of ideal, $J$ is $\lambda^{B_1}$-invariant and $\rho^{B_1}$-invariant in~$B_1$. Thus, for every $b\in B_1$, $r_{B_1}(bJ, J) = J \times bJ$ and $r_{B_1}(J, bJ) = bJ \times J$.
\end{exemple}

However, it is possible to construct i-homomorphisms between solutions associated with skew braces that are not homomorphisms of skew braces. A left ideal $S$ of a skew brace $B$ is said to be a \emph{strong left ideal} if it is normal in the additive group of~$B$.

\begin{proposicio}
\label{prop:strong-lr-invariant}
Let $B$ be a skew brace and let $S$ be a strong left ideal of~$B$. Then, $S$ is $\lambda$ and $\rho$-invariant.
\end{proposicio}

\begin{proof}
Let $b\in B$ and $s \in S$. Since $S$ is a left ideal, it remains to check that $\rho_b(s) = (s^{-1}+b)^{-1}b \in S$. Observe that
\[
 \rho_b(s)^{-1} = b^{-1}(s^{-1}+b) = b^{-1}s^{-1} - b^{-1} = b^{-1} + \lambda_{b^{-1}}(s^{-1}) - b^{-1}\]
which belongs to $S$ because $S$ is $\lambda$-invariant and a normal subgroup of the additive group in~$B$. Then, the proposition holds as $S$ is closed under inverses of both products and sums.
\end{proof}

\begin{exemple}
\label{ex:i-morfism-no-morfism}
Let $B$ be a skew brace and let $S$ be a strong left ideal of~$B$. By Proposition~\ref{prop:strong-lr-invariant}, we know that $S$ is $\lambda$ and $\rho$-invariant. Then, $\mathcal{K} = \{S, B\setminus S\}$ defines a partition of $B$ such that $\lambda_X(Y) = Y$ and $\rho_Y(X) = X$ for every $X,Y \in \mathcal{K}$. Take $B_2 = \{0,b\}$ a trivial skew brace of order $2$, and define $f\colon B \rightarrow B_2$ as $f(x) = 0$ if $x \in S$, and $f(x) = b$ otherwise. It turns out that $f$ is an i-homomorphism of skew braces with i-kernel~$S$.
\end{exemple}

Ideals of skew braces are subsets of a solution with stronger properties than i-kernels; but the following example shows that, when embedded in associated structure skew braces, they are no longer ideals in general.

\begin{exemple}
\label{ex:InoidG(X,r)}
Consider $B$ as the \texttt{SmallSkewbrace}(8,47) in the \texttt{YangBaxter} library of GAP~\cite{GAP4-15}. We can see that $B$ is of abelian type, it satisfies that $\Ker \lambda = \{0\}$, and it has a unique proper ideal $0 \lneq I \lneq B$, which is an abelian ideal of order~$4$. Thus, $B$ is a soluble skew brace. Recall that we denote $\iota \colon b \in B \mapsto \bar b\in G_B$. In this example, we consider only $\lambda$ and $\rho$ maps of~$B$.

Let us call $\bar I = \iota(I) = \{0,\bar a_1,\bar a_2,\bar a_3\}$ and $\overline{B \setminus I} = \iota(B \setminus I) = \{\bar b,\bar c,\bar d,\bar e\}$. We know that 
\[ (G_B,\cdot) = \langle B \mid \bar x \bar y = \overline{\lambda_x(y)}\,\overline{\rho_y(x)}, \ \forall\,x,y\in B\rangle.\]
Let us compute all relations in $(G_B,\cdot)$. Since $I$ is an abelian skew brace, it means that $\left.\lambda_{a_i}\right|_I = \operatorname{id}_I = \left.\rho_{a_i}\right|_I$, $1\leq i \leq 3$. Thus, $\langle \bar I \rangle_{\bigcdot}$ is an abelian subgroup of $(G_B,\cdot)$. We also know that $\bar 0 \in G_B$ is an element in the centre of $(G_B,\cdot)$. Moreover, it holds that
\[ \lambda_x(a_1) = \rho_x(a_1) = a_3, \quad \lambda_x(a_2) = \rho_x(a_2) = a_2,\quad  \lambda_x(a_3) = \rho_x(a_3) = a_1\]
for every $x \in B\setminus I$. Since $B$ is of abelian type, $r_B$ is involutive, and therefore, for every $x \in B \setminus I$, $\bar x\bar a_1 = \bar a_3 \bar y$ for some $y \in B\setminus I$ implies that $\lambda_{a_3}(y) = x$. Thus, $\bar x\bar a_1 = \bar a_3\bar y$ appears twice in the set of relations. We argue analogously for $\bar a_1\bar x = \bar y\bar a_3$. Therefore, the set of relations with elements of $I$ are the following ones
\begin{equation}
\label{eq:rel1}
\begin{array}{lll}
\bar c\bar a_1 = \bar a_3\bar e & \bar b\bar a_2 = \bar a_2\bar e & \bar b\bar a_3 = \bar a_1\bar c\\
\bar b\bar a_1 = \bar a_3\bar d & \bar c\bar a_2 = \bar a_2\bar d & \bar c\bar a_3 = \bar a_1\bar b\\
\bar d\bar a_1 = \bar a_3\bar b & \bar d\bar a_2 = \bar a_2\bar c & \bar d\bar a_3 = \bar a_1\bar e\\
\bar e\bar a_1 = \bar a_3\bar c & \bar e\bar a_2 = \bar a_2\bar b & \bar e\bar a_3 = \bar a_1\bar d
\end{array}
\end{equation}
On the other hand, it also holds $\lambda_b(b) = \rho_b(b) = b$; $\lambda_c(d) = c$, $\rho_d(c) = d$; $\lambda_e(e) = \rho_e(e) = e$ which provide empty relations. In addition, for every $x,y\in B\setminus I$ if $\lambda_x(y) = u$ and $\rho_y(x) = v$ for certain $u, v\in B\setminus I$, by the involutive property of~$r$, $\lambda_u(v) = x$ and $\lambda_v(u) = y$. Therefore, we complete the set of relations with the elements in $B\setminus I$:
\begin{equation}
\label{eq:rel2}
\begin{array}{lll}
\bar b\bar c = \bar c\bar e & \bar b\bar d = \bar e\bar c & \bar b\bar e = \bar d\bar d\\
\bar c\bar b = \bar d\bar e & \bar c\bar c = \bar e\bar b & \bar d\bar b = \bar e\bar d
\end{array}
\end{equation}
Call $R$ the set of relators in $(G_B,\cdot)$. We conclude that $w \in R$ if, and only if, either $w$ is a commutator of the form $[\bar x,\bar y] = \bar x\bar y\bar x^{-1}\bar y^{-1}$ where $x,y \in I$, $w = [\bar 0,\bar x] = \bar 0\bar x\bar 0^{-1}\bar x^{-1}$ where $x \in B$, or $w = \bar x\bar y\bar z^{-1}\bar t^{-1}$ comes from one of the previous relations in~\eqref{eq:rel1} or~\eqref{eq:rel2}. Take $N$ the normal subgroup generated by~$R$.

We show that $\langle \bar I \rangle_{\bigcdot}$ is not an ideal of $G(X,r)$, since it is not a normal subgroup of $(G_B,\cdot)$. Indeed, take $b \in B\setminus I$, and by a way of contradiction assume that $\bar b\bar a\bar b^{-1}\in \langle I \rangle_{\bigcdot}$. Then, there exist $a_1, \ldots, a_n \in I$, and $\varepsilon_i\in \{\pm 1\}$ with $1\leq i \leq n$, such that the word
\[ w:= \bar b \bar a\bar b^{-1}\bar a_1^{\varepsilon_1}\cdots \bar a_n^{\varepsilon_n} \]
belongs to $N$, i.e. $w$ is a finite product of conjugates of elements in either $R$ or~$R^{-1}$. Observe that the letter $\bar b^{-1}$ in the word $w$ can not be a conjugator element of neither an element of $R$ nor $R^{-1}$ because, after it on the word~$w$, there are only products of elements in $I\cup I^{-1}$. Hence, $\bar b\bar a\bar b^{-1}$ is a subword of either a relator $r\in R$ or $r^{-1}\in R^{-1}$. According to relations in~\eqref{eq:rel1} and~\eqref{eq:rel2}, that is not possible, and we arrive to a contradiction.
\end{exemple}

Nevertheless, we can see that every i-kernel of an i-homomorphism induces both a strong left ideal and an ideal of the structure skew brace of a solution.

\begin{nota}
\label{nota:imatge-i-kernel}
Let $f\colon X \rightarrow Y$ be an i-epimorphism between solutions $(X,r)$ and $(Y,s)$, and let $X_0\subseteq X$ be an i-kernel of $f$ such that $y_0:= f(x)$ for every $x\in X_0$. Since $f$ is a homomorphism, $f$ makes compatible the $\lambda$ and $\rho$ components of $r$ and~$s$. Moreover, $r(Z,X_0) = X_0\times Z$ and $r(X_0,Z) = Z \times X_0$ for every equivalence class of $\ker f$. Since $f$ is surjective, it follows that $\lambda_{y_0}^s = \rho_{y_0}^s = \mathrm{id}_Y$, and $y_0$ is a fixed element by $\lambda^s$'s and $\rho^s$'s, i.e. $\lambda^s_y(y_0) = \rho^s_y(y_0) = y_0$ for every $y\in Y$. Hence, according to relations~\eqref{eq:G(X,r)prod} and~\eqref{eq:G(X,r)sum} of the multiplicative and additive group of the structure skew brace of a solution, it follows that $y_0 = \iota(y_0)\in G_U$ is a central element.
\end{nota}

\begin{teorema}
\label{teo:X0-kerf-ideal}
Let $f\colon X \rightarrow U$ be an i-homomorphism of solutions $(X,r)$ and $(U,s)$ with i-kernel~$X_0$, and induced homomorphism $\bar{f}\colon G_X \rightarrow G_U$. Then, it follows that
\begin{enumerate}
\item $\langle X_0 \rangle_{\bigcdot} = \langle X_0 \rangle_+$ is a strong left ideal of $G_X$;
\item $I^f_{X_0}:= \langle X_0\rangle_{\bigcdot}\Ker \bar{f} = \langle X_0\rangle_+ + \Ker \bar{f}$ is an ideal of~$G_X$;
\item \label{item:x(kerf)y} for every $x,y \in X$ with $x (\ker f) y$, $xI^f_{X_0} = yI^f_{X_0}$;
\item if $\iota\colon U \rightarrow G_U$ is injective, then converse of item~\ref{item:x(kerf)y} is also true.
\end{enumerate}
\end{teorema}

\begin{proof}
1. Since $(X_0, \left. r\right|_{X_0})$ is a subsolution, then $\langle X_0\rangle_{\bigcdot} = \langle X_0 \rangle_+$ is a subbrace of $G_X$. By the definition of i-homomorphism, for every $x\in X$ it holds that $\lambda^r_x(X_0) = (\lambda^r)_x^{-1}(X_0) = X_0$. Thus, applying equation~\eqref{eq:lambda-en-GX}, $\lambda^{G_X}_w(X_0) = X_0$ for every $w \in G_X$, and therefore, $\langle X_0\rangle_+$ is $\lambda^{G_X}$-invariant. Moreover, $\langle X_0 \rangle_+$ is a normal subgroup of 
\[ (G_X,+) = \langle X \mid x + \lambda_x(y) = \lambda_x(y) + \lambda_{\lambda_x(y)}(\rho_y(x)),\, \forall x,y\in X \rangle_+.\]
Indeed, let $x \in X$ and $z\in X_0$. Since $\lambda_z$ is bijective, there exists $y\in X$ such that $\lambda_z(y) = x$. Then, it holds that
\[ z + x = z + \lambda_z(y) = \lambda_z(y) + \lambda_{\lambda_z(y)}(\rho_y(z)).\]
By the definition of i-homomorphism, $\rho_y(z)\in X_0$ so that $\lambda_{\lambda_z(y)}(\rho_y(z))\in X_0$. Hence, $-x + z + x \in X_0$. On the other hand, take $v \in X_0$ such that $\lambda_x(v) = z$. Since $r$ is bijective (see~\cite[Corollary 3.4]{JedlickaPilitowska26}), there exist $w, y \in X$  such that $\lambda_w(y) = x$ and $\rho_y(w) = v$. By the definition of i-homomorphism, $w \in X_0$, and therefore
\[ x + z = \lambda_w(y) + \lambda_{\lambda_w(y)}(\rho_y(w)) = w + \lambda_w(y) = w + x.\]
Thus, $x + z - x = w \in X_0$. Hence, $\langle X_0\rangle_+$ is a strong left ideal.

2. Since $\Ker \bar{f}$ is an ideal of~$G_X$, it follows that
\[ I^f_{X_0}:= \langle X_0\rangle_{\bigcdot}\Ker \bar{f} = \langle X_0\rangle_+ + \Ker \bar{f}\]
Moreover, since $\langle X_0\rangle_+$ is a strong left ideal, it suffices to prove that $I^f_{X_0}$ is a normal subgroup of the multiplicative group 
\[ G_X = \langle X \mid xy = \lambda_x(y)\rho_y(x), \, \forall x, y\in X \rangle.\]
Let $x \in X$ and let $z \in X_0$. By the definition of i-homomorphism, it holds that $\lambda_z(x) (\ker f) x$. Thus, by~\eqref{eq:x(ker f)y->Kerf}, there exists $w \in \Ker \bar{f}$ such that $\lambda_z(x) = xw \in G_X$. On the other hand, $\rho_x(z) \in X_0$. Then, it holds that 
\[ zx = \lambda_z(x)\rho_x(z) = xw\rho_x(z),\]
and therefore, $x^{-1}zx = w \rho_x(z) \in I^{f}_{X_0}$. Similarly, $\rho_z(x) (\ker f) x$ so that $\rho_z(x) = wx$, for some $w \in \Ker\bar{f}$, and $\lambda_x(z) \in X_0$. Thus, 
$xzx^{-1} = \lambda_x(z)w \in I^f_{X_0}$. Hence, $I^f_{X_0}$ is a normal subgroup of $(G_X,\cdot)$, and $I^f_{X_0}$ is an ideal.

3. If $x, y\in X$ satisfy that $f(x) = f(y)$, then $\bar{f}(x) = \bar{f}(y)$ and $xI^f_{X_0} = yI^f_{X_0}$. 

4. On the other hand, suppose that $xI^f_{X_0} = yI^f_{X_0}$, for some $x,y\in X$. Without loss of generality, we can assume that $f$ is surjective. It suffices to show that $f(x) = f(y) \in G_U$, as $\iota \colon U \rightarrow G_U$ is injective.

Call $u_0:= f(X_0) \in U$. By definition of $I^f_{X_0}$, it holds that
\[ x^{-1}y = z_1^{\varepsilon_1}\cdots z_n^{\varepsilon_n}k,\]
for some $k\in \Ker \bar{f}$, $z_1, \ldots, z_n \in X_0$ and $\varepsilon_i = \pm 1$, $1\leq i \leq n$. Since $u_0 = f(z_i)$ for every $1\leq i \leq n$, it follows that 
\[ f(x)^{-1}f(y) = \bar{f}(x)^{-1}\bar{f}(y) = \bar{f}(x^{-1}y) = u_0^{\varepsilon_1+\cdots + \varepsilon_n} \in G_Y.\]
Call $u_1 := f(x) \in U$ and $u_2 := f(y)\in U$. By the previous remark, $u_0$ is a central element in $G_U$. Thus, it follows that all relators of $G_U$ with the letter $u_0$ are of the form $u_0uu_0^{-1}u^{-1}$ for every $u\in U$. Therefore, 
\[ u_1^{-1}u_2 = u_0^{\varepsilon_1+ \cdots + \varepsilon_n} \in G_U \Leftrightarrow u_1^{-1}u_2 = 0\in G_U\]
i.e. $f(x) = f(y) \in G_U$.
\end{proof}

\begin{corolari}
\label{cor:I_J-ideal-brace}
Let $B$ be a skew brace with an ideal $J$. Let $f\colon B \rightarrow B/J$ be a canonical epimorphism with induced epimorphism $\bar{f}\colon G_B \rightarrow G_{B/J}$. The following statements hold:
\begin{enumerate}
\item $I_J^f:= \langle \bar{J} \rangle_{\bigcdot}\Ker \bar f$ is an ideal of $G_B$ such that $\bar b_1 I_J^f = \bar b_2 I_J^f$ if and only if $b_1J = b_2J$, for every $b_1,b_2\in B$;

\item $\pi_B(I_J) = J$, where $I_J = I_J^f\Ker \pi_B$, with $\pi_B$ the projection epimorphism of $G_B$ onto~$B$;
\end{enumerate}
\end{corolari}

\begin{proof}
Every homomorphism of skew braces is an i-homomorphism of associated solutions, and every associated solution with a skew brace is injective. Thus, the first statement is a direct consequence of Theorem~\ref{teo:X0-kerf-ideal}.


Let  $I_J = I_J^f\Ker \pi_B$, which is an ideal of~$G_B$ as it is a product of ideals. Take $\pi_{B/J} \colon G_{B/J} \rightarrow B/J$ the projection epimorphism of $G_{B/J}$ onto~$B/J$. By the definition of $\bar{f}$, $\pi_B$ and $\pi_{B/J}$ the following diagram commutes in both directions
\[ \xymatrix{B \ar@/^/[d]^{\iota} \ar[r]^{f}  & B/J   \ar@/^/[d]^{\bar{\iota}} \\ 
G_B \ar@/^/[u]^{\pi_B}  \ar[r]_{\bar f} & G_{B/J} \ar@/^/[u]^{\pi_{B/J}}}\]
Indeed, let $w = \bar b_1^{\varepsilon_1}\cdots \bar b_n^{\varepsilon_n} \in G_B$, with $b_1, \ldots , b_n \in B$, and $\varepsilon_i = \pm 1$, $1\leq i \leq n$. It follows that
\begin{align*}
\pi_{B/J}(\bar{f}(w)) & = \pi_{B/J}\left(\bar{f}(\bar b_1^{\varepsilon_1}\cdots \bar b_n^{\varepsilon_n})  \right) = \pi_{B/J}\left(\overline{f(b_1)}\,^{\varepsilon_1}\cdots \overline{f(b_n)}\,^{\varepsilon_n}\right) \\
& = f(b_1)^{\varepsilon_1}\cdots f(b_n)^{\varepsilon_n} \\
f(\pi_B(w)) & = f\left(\pi_B(\bar b_1^{\varepsilon_1}\cdots \bar b_n^{\varepsilon_n})\right) = f(b_1^{\varepsilon_1}\cdots b_n^{\varepsilon_n}) \\
& = f(b_1)^{\varepsilon_1}\cdots f(b_n)^{\varepsilon_n}
\end{align*}

The inclusion $J \subseteq \pi_B(I_J)$ follows from the fact that $J = \pi_B(\iota(J))$ with $\iota(J) = \bar{J} \subseteq I_J$. Now, let $\pi_B(w) \in B$ for some $w\in I_J$, and write $w = ykl$, for some $y \in \langle \bar{J} \rangle_{\bigcdot}$, $k\in \Ker\bar{f}$, and $l\in \Ker \pi_B$. By the commutativity of the previous diagram, $0 = f(\pi_B(l)) = \pi_{B/J}\bar{f}(l)$, and therefore, it follows that
\[ f(\pi_B(w)) = \pi_{B/J}(\bar{f}(ykl)) = \pi_{B/J}(\bar{f}(y)) \]
Write $y = \bar j_1^{\varepsilon_1}\cdots \bar j_r^{\varepsilon_r} \in G_B$ for some $j_i\in J$ and $\varepsilon_i = \pm 1$, $1\leq i \leq r$. By the definition of~$\bar f$, 
\[ \bar{f}(y) = \overline{f(j_1)}\,^{\varepsilon_1}\cdots \overline{f(j_r)}\,^{\varepsilon_r} = \bar{0}^{\varepsilon_1}\cdots \bar 0^{\varepsilon_r} \in G_{B/J}\]
Since $\pi_{B/J}(\bar 0) = 0\in B/J$, it follows that $\pi_{B/J}(\bar{f}(y)) = 0\in B/J$, and therefore, $f(\pi_B(w)) = 0\in B/J$, i.e $\pi_B(w) \in J$.
\end{proof}

%

\section{Soluble solutions of the YBE}
\label{sec:soluble}
Soluble skew braces were defined in~\cite{BallesterEstebanJimenezPerezC24-solubleskewbraces} as a class of skew braces with a rich ideal structure, in contrast with simple skew braces. In this section, we introduce a notion of solubility of solutions by means of the previous definition of i-homomorphisms of solutions. This will give rise to a class of solutions with a rich decomposability, in contrast with indecomposable solutions.


According to the definition of solubility of skew braces in~\cite{BallesterEstebanJimenezPerezC24-solubleskewbraces}, a soluble skew brace $B$ can be characterised by the existence of skew braces $\{B_k\}_{k=0}^t$ and epimorphisms $f_k \colon B \rightarrow B_k$ such that
\begin{align}
& \Ker f_t = 0 \subseteq \Ker f_{t-1} \subseteq \ldots \subseteq \Ker f_1 \subseteq \Ker f_0 = B; \label{eq:cadenaideals}\\
& \text{$f_k(\Ker f_{k-1})$ is an abelian ideal $J$ of $B_{k}$ for each $1\leq k \leq t$.}  \label{eq:abelianideal}
\end{align}

\begin{definicio}
\label{def:soluble-solution}
Let $(X,r)$ be a solution and $\{B_k\}_{k=0}^t$ be a sequence of skew braces. We say that $(X,r)$ is \emph{soluble}, if there exist i-epimorphisms of solutions $f_k \colon X \rightarrow B_k$, with respective i-kernels $X_k \subseteq X$ for each $0\leq k \leq t-1$, and an epimorphism of solutions $f_t\colon X \rightarrow B_t$, satisfying that for every $1\leq k \leq t$:
\begin{enumerate}
\item $\ker f_{k} \subseteq \ker f_{k-1}$, $\ker f_0 = \{X\}$ and $\ker f_t \subseteq \ker \iota$, where $\iota \colon X \rightarrow G_X$ is canonical homomorphism;
\item $f_{k}(X_{k-1})$ is contained in an abelian ideal $J$ of $B_k$, with $f_k(x)J = f_k(y)J$, for every $x,y\in X$ such that $x (\ker f_{k-1}) y$.
\end{enumerate}
\end{definicio}

Our first main result shows that soluble solutions provide soluble structure skew braces.  

\begin{teorema}
\label{teo:soluble-solution->GX-soluble}
Let $(X,r)$ be a soluble solution. Then, its associated structure skew brace is soluble.
\end{teorema}

\begin{proof}
Let $\{B_k\}_{k=0}^t$ be a sequence of skew braces such that there exist an i-epimorphism $f_k \colon X \rightarrow B_k$ with $i$-kernel $X_k \subseteq X$, for each $0\leq k \leq t-1$, and an epimorphism $f_t\colon X \rightarrow B_t$ with $\ker f_t \subseteq \ker \iota$, satisfying conditions 1 and~2 of the definition.

According to Theorem~\ref{teo:X0-kerf-ideal}, for every $0\leq k \leq t-1$, we can take 
\[ I_k:= I_{X_k}^{f_k} = \langle X_k \rangle_{\bigcdot}\Ker \bar{f}_k,\]
which is an ideal of $G_X$, where $\bar f_k \colon G_k \rightarrow G_{B_k}$ is an epimorphism induced by~$f_k$.

Fix $1\leq k \leq t$. Let $J$ be an abelian ideal of $B_k$ such that $f_k(X_{k-1})$ is contained in~$J$, and $J$ satisfies item $2$ of definition of soluble solution. Let $g\colon B_k \rightarrow B_k/J$ be a canonical epimorphism. According to Corollary~\ref{cor:I_J-ideal-brace},  $I_J:= \langle \bar J\rangle_{\bigcdot}\Ker \bar{g} \Ker \pi_{B_k}$ is an ideal of~$G_{B_k}$, where $\bar{g}\colon G_{B_k}\rightarrow G_{B_k/J}$ is an epimorphism induced by~$g$, such that $I_J/\Ker\pi_{B_k}$ is isomorphic to~$J$. Since $\Ker \pi_{B_k}$ is an abelian ideal of $G_{B_k}$ by Lemma~\ref{lema:kerg-socle}, it follows that $I_J$ is a soluble ideal with respect to~$G_{B_k}$.

Now, by the definition of $\bar{f}_k$, it holds that $\bar{f}_k(\langle X_{k-1} \rangle_{\bigcdot}) \subseteq \langle \bar J \rangle_{\bigcdot} \subseteq G_{B_k}$ as $f_k(X_{k-1}) \subseteq J$. Moreover, for every $x,y\in B$ with $f_{k-1}(x) = f_{k-1}(y)$, it holds that $f_{k-1}(x)J = f_{k-1}(y)J$, and therefore, $g(f_{k-1}(x)) = g(f_{k-1}(y))$. Thus, by Lemma~\ref{lema:kerbarf}, 
it follows that $\bar{f}_{k}(\Ker \bar f_{k-1})$ is contained in $\Ker \bar g$. Therefore, 
\[ \bar f_k(I_{k-1}) = \bar f_k(\langle X_{k-1}\rangle_{\bigcdot}\Ker \bar f_{k-1})\subseteq I_J.\]
Hence, $I_{k-1}/\Ker \bar f_k$ is a soluble ideal with respect to $G_X/\Ker \bar f_k$. In particular, for every $1 \leq k \leq t-1$, $I_{k-1}/I_k$ is a soluble ideal with respect to $G_X/I_k$, and $I_{t-1}/\Ker \bar f_t$ is a soluble ideal with respect to $G_X/\Ker \bar f_t$. Thus, $G_X/\Ker \bar f_t$ is a soluble skew brace.

Finally, since $\ker f_t \subseteq \ker \iota$, by Lemma~\ref{lema:kerbarf}, observe that $\Ker \bar f_t$ is the ideal generated by 
\[ \{x^{-1}y\in G_X\mid x(\ker f_t) y\} = \{x^{-1}y\in G_X\mid \iota(x) = \iota(y)\} = \{0\}\]
Hence, $\Ker \bar f_t = 0$ so that $G_X$ is a soluble skew brace.
\end{proof}

As a consequence, solubility of solutions extends the notion of solubility of skew braces. Before, we recall the definition of permutation skew brace of a solution $(X,r)$. Consider the group 
\[ \mathcal{G}(X,r) = \langle (\lambda_x, \rho_x^{-1}) \mid x \in X \rangle \leq \rm{Sym}_{X}\times \rm{Sym}_X\]
Following~\cite[Definition 1.5]{CedoJespersKubatVanAntwerpenVerwimp23}, consider the group homomorphism 
\[ h\colon (G_X,\cdot) \rightarrow \mathcal{G}(X,r), \quad \text{given by} \quad x \mapsto (\lambda_x, \rho_x^{-1}).\]
It turns out that $\Ker h$ is an ideal of $G_X$ contained in $\Soc(G_X)$, and therefore, we can define a sum in $\mathcal{G}(X,r)$ so that $(\mathcal{G}(X,r),+,\cdot)$ is a skew brace isomorphic to $G_X/\Ker h$, known as the \emph{permutation skew brace} of $(X,r)$ (see~\cite{CedoJespersKubatVanAntwerpenVerwimp23} for a detailed construction). Observe that $\mathcal{G}(X,r)$ is finite whenever $X$ is a finite set.
\begin{corolari}
\label{cor:soluble-brace-solution}
Let $B$ be a skew brace. The following are equivalent:
\begin{enumerate}
\item $B$ is a soluble skew brace;
\item $(B, r_B)$ is a soluble solution;
\item $G_B$ is a soluble skew brace;
\item $\mathcal{G}(B,r_B)$ is a soluble skew brace.
\end{enumerate}
\end{corolari}

\begin{proof} 
$1$ implies $2$ is given by Example~\ref{ex:homomorfisme-brides} and the characterisation of solubility of skew braces given by~\eqref{eq:cadenaideals} and~\eqref{eq:abelianideal}. 

$2$ implies $3$ is given by Theorem~\ref{teo:soluble-solution->GX-soluble}.

$3$ implies $1$ is clear as $B$ is isomorphic to $G_B/\Ker \pi_B$, where $\pi_B\colon G_B \rightarrow B$ is the projection epimorphism onto~$B$.

$3$ if, and only if, $4$ follows from the fact that $\Ker h \subseteq \Soc(G_B)$ is an abelian ideal of $G_B$, where $h\colon G_B \rightarrow \mathcal{G}(B,r_B)$ is the previous homomorphism.
\end{proof}

The converse of Theorem~\ref{teo:soluble-solution->GX-soluble} is not true in general (see Example~\ref{ex:multipermut-nosoluble} below). Nevertheless, we give the following sufficient condition to get the converse of Theorem~\ref{teo:soluble-solution->GX-soluble}. This is the new version of Theorem~D in~\cite{BallesterEstebanJimenezPerezC24-solubleskewbraces} for the new definition of soluble solutions.

\begin{teorema}
\label{teo:G_X-soluble+xinI->Xsoluble}
Let $(X,r)$ be a solution with soluble structure skew brace $G_X$. If there exists an abelian ideal $I$ of $G_X$ such that $\iota(x) \in I$ for some $x\in X$, then $(X,r)$ is a soluble solution.
\end{teorema}

\begin{proof}
Write $G:= G_X$. Since $G$ is a soluble skew brace and $I$ is an abelian ideal of $G$, we can consider an ideal series
\[ 0 = I_t \subseteq I_{t-1} \subseteq \ldots \subseteq I_1 \subseteq I_0 = G,\]
with $I_{t-1} = I$, and $I_{k-1}/I_k$ an abelian skew brace for every $1 \leq k \leq t$. For every $0\leq k \leq t-1$ take 
\[ X_{k} := \{x \in X \mid \iota(x) \in I_k\} \neq \emptyset \]
as $\iota(x) \in X_{t-1} \neq \emptyset$. Thus, we obtain a chain of subsets
\[ X_{t-1} \subseteq \ldots \subseteq X_1 \subseteq X_0 = X.\]

Fix $0 \leq k \leq  t$. Recall that  $\iota \colon X \rightarrow G$ satisfies that $r_G(\iota \times \iota) = (\iota\times \iota)r$, where $(G,r_G)$ is the associated solution to the skew brace structure~$G$. Then, the map $f_k\colon X \rightarrow G/I_k$, given by $f_k(x) = \iota(x)I_k$ is a homomorphism of solutions $(X,r)$ and $(G/I_k, r_{G/I_k})$. In particular, $f_t = \iota \colon X \rightarrow G_X$. 

Suppose that $0 \leq k \leq t-1$. For every $x \in X_k$,
\[ \{y \in X\mid f_k(x) = f_k(y)\} = \{y \in X \mid I_k = \iota(y)I_k\} =  X_k,\]
i.e. $X_k$ is an equivalence class of $\ker f_k$. Moreover, since $I_k$ is an ideal of~$G$, it is $\lambda^G$ and $\rho^G$-invariant, and $\lambda^G_w(w')I_k = w'I_k$ and $\rho^G_w(w') = w'I_k$ for every $w\in I_k$ and $w'\in G$. Therefore, it follows that for every equivalence class $Z$ of the relation $\ker f_k$,
\begin{align*}
\iota(\lambda_z(x))I_k & = \lambda_{\iota(z)}^G(\iota(x)) I_k  = I_k\\
\iota(\lambda_x(z))I_k & = \lambda_{\iota(x)}^G(\iota(z)) I_k = \iota(z)I_k\\
\iota(\rho_z(x))I_k & = \rho_{\iota(z)}^G(\iota(x)) I_k  = I_k\\
\iota(\rho_x(z))I_k & = \rho_{\iota(x)}^G(\iota(z)) I_k = \iota(z)I_k
\end{align*}
for every $x\in X_k$ and $z\in Z$. Thus, $r(X_k,Z) = Z \times X_k$ and $r(Z, X_k) = X_k \times Z$ for every equivalence class $Z$ of $\ker f_k$. Hence, $f_k$ is an i-homomorphism and $X_k$ is an i-kernel of~$f_k$.

Clearly, for every $1\leq k \leq t$, $\ker f_k \subseteq \ker f_{k-1}$, $\ker f_0 = \{X\}$ and $\ker f_t = \ker \iota$. In addition, if $1 \leq k \leq t$, then 
\[ f_k(X_{k-1}) = \{\iota(x)I_k \mid x \in X_{k-1}\} = \{\iota(x)I_k\mid \iota(x) \in I_{k-1}\} \subseteq I_{k-1}/I_k,\]
which is an abelian ideal of $G/I_k$, and for every $x,y\in X$, 
\[ f_k(x)I_{k-1}/I_k = f_k(y)I_{k-1}/I_k \Leftrightarrow \iota(x)I_{k-1} = \iota(y)I_{k-1}\Leftrightarrow  f_{k-1}(x) = f_{k-1}(y).\]
Hence, we conclude that $(X,r)$ is a soluble solution.
\end{proof}

Theorem~\ref{teo:G_X-soluble+xinI->Xsoluble} opens the door to explore the relation between soluble solutions and multipermutation solutions. Let $(X,r)$ be a solution, and consider the equivalence relation $\sim$ on $X$ defined as $x\sim y$ if $\lambda_x = \lambda_y$ and $\rho_x = \rho_y$. The so-called \emph{retraction solution} $\Ret(X,r) = (X/\sim,\bar{r})$ arises, where $\bar{r}$ is defined by
\[ \bar{r}([x], [y]) = ([\lambda_x(y)], [\rho_y(x)]), \quad \text{for all $[x], [y] \in X/ \sim$}.\] 
We can iterate this process and define inductively
\begin{align*}
\Ret^1(X, r) & = \Ret(X,r),\\
\Ret^{n+1}(X,r) & = \Ret(\Ret^n(X,r)), \quad \text{for all $n\geq 1$.}
\end{align*}
A solution $(X,r)$ is said to be a \emph{multipermutation solution of level} $m$, if $m$ is the smallest natural such that $\Ret^m(X,r)$ is a singleton solution. A singleton solution is multipermutation of level~$0$. It is well-known that a solution is multipermutation if, and only if, $G_X$ and $\mathcal{G}(X,r)$ are right nilpotent of nilpotent type, and in addition, right nilpotent skew braces are soluble (cf.~\cite{BallesterEstebanJimenezPerezC24-solubleskewbraces}).

\begin{corolari}
\label{cor:X-lambda_x=rho_x=id->soluble}
Let $(X,r)$ be a solution with soluble structure skew brace~$G_X$. If there exists $x\in X$ such that $\lambda_x = \rho_x = \mathrm{id}_X$, then $(X,r)$ is a soluble solution.
\end{corolari}

\begin{proof}
Observe that if $\lambda_x = \rho_x = \mathrm{id}_X$, then by equations~\eqref{eq:lambda-prod-GX} and~\eqref{eq:rho-prod-GX}, it turns out that $\lambda_{\iota(x)}^{G_X} = \rho_{\iota(x)}^{G_X} = \mathrm{id}_{G_X}$. Thus, $\iota(x) \in \Soc(G_X)$, and $\Soc(G_X)$ is an abelian ideal of $G_X$. Theorem~\ref{teo:G_X-soluble+xinI->Xsoluble} yields the result.
\end{proof}

\begin{corolari}[{\cite[Corollary~3]{BallesterEstebanJimenezPerezC24-solubleskewbraces}}]
\label{cor:sol_mult->sol_sol}
Let $(X,r)$ be a multipermutation solution. If $\Soc(G_X) \cap \iota(X) \neq \emptyset$, then $(X,r)$ is a soluble solution.
\end{corolari}

The following example shows that condition $\iota(x) \in I$, with $I$ an abelian ideal of $G_X$ is essential, as some multipermutation solutions can be simple.
\begin{exemple}[{\cite[Example~4]{BallesterEstebanJimenezPerezC24-solubleskewbraces}}]
\label{ex:multipermut-nosoluble}
Consider the set $X = \{1,2,3\}$ and the solution $(X,r)$ of previous Example~\ref{ex:contraex-Xsimple-GXnosimple}. Since it is a Lyubashenko solution,  $(X,r)$ has multipermutational level $1$, as $\Ret(X,r)$ is a singleton solution. Therefore, $G_X$ is right nilpotent, and therefore, a soluble skew brace. Nevertheless, we have seen that $(X,r)$ is a simple solution which can not be soluble.
\end{exemple}

\section*{Acknowledgements}
The first, second, and fourth authors are supported by the Ministerio de Ciencia, Innovación y Universidades, Agencia Estatal de Investigación, FEDER (grant: PID2024-159495NB-I00), and the Conselleria d'Educació, Universitats i Ocupació, Generalitat Valenciana (grant: \mbox{CIAICO/2023/007}).

\end{document}